\documentclass[11pt]{amsart}

\usepackage{latexsym}
\usepackage{amsfonts}
\usepackage{amssymb}
\usepackage{fullpage}
\usepackage[OT4]{fontenc}
\usepackage{amscd}
\usepackage[all]{xy}
\usepackage{srcltx}

\def\cal{\mathcal}

\newcommand{\field}[1]{\mathbb{#1}}
\newcommand{\C}{\field{C}}

\newcommand{\N}{\field{N}}

\renewcommand{\P}{\mathbb{P}}

\newcommand{\gen}{\mathrm{gen}}
\newcommand{\ity}{\infty}
\newcommand{\m}{\setminus}

\newtheorem{theo}{Theorem}[section]

\newtheorem{lem}[theo]{Lemma}
\newtheorem{co}[theo]{Corollary}

\theoremstyle{definition}
\newtheorem{defi}[theo]{Definition}
\theoremstyle{remark}
\newtheorem{re}[theo]{\sc Remark}

\renewcommand{\min}{{\rm{min}}}
\renewcommand{\inf}{{\rm{inf}}}
\newcommand{\ord}{{\rm{ord}}}
\newcommand{\Sing}{{\rm{Sing}}}
\renewcommand{\graph}{{\rm{graph}}}

\newtheorem*{conj*}{Conjecture}

\newcommand{\fin}{\hspace*{\fill}$\Box$\vspace*{2mm}}

\font\tenmsy=msbm10

\def\Bbb#1{\hbox{\tenmsy#1}} 

\title[Bifurcation at infinity]
{Bifurcation locus and branches at infinity of a polynomial $f:\C^2\to \C$ }
\makeatletter

\@addtoreset{equation}{section}
\makeatother
\author{Zbigniew Jelonek}
\address{Instytut Matematyczny,
Polska Akademia Nauk,
\'Sniadeckich 8, 00-956 Warszawa\\
Poland.}
\email{najelone@cyf-kr.edu.pl}

\author{Mihai Tib\u ar}
\address{Math\' ematiques, UMR 8524 CNRS,
Universit\'e Lille 1, \  59655 Villeneuve d'Ascq, France.}
\email{tibar@math.univ-lille1.fr}
\thanks{The first author was partially supported by  Universit\'e Lille 1   and by
the grant of NCN, 2014-2017}

\date{\today}

\begin{document}

\begin{abstract}
We show that the number of bifurcation points at infinity of a
polynomial function  $f:\C^2\to\C$ is at most the number of
branches at infinity of a generic fiber of $f$ and that this upper bound can be diminished
by one in certain cases.
\end{abstract}

\maketitle

\section{Introduction}
Let  $f:\C^2 \to \C $ be a polynomial function in a fixed coordinate system. It is well known (as being proved originally by Thom \cite{T}), that $f$ is a locally trivial $C^\infty$ fibration outside  a
 finite subset  of the target. The smallest such set is called
{\it the bifurcation set of  $f$} and will be denoted here  by
$B(f)$. The set $B(f)$ might be larger than the set $f(\Sing f)$
of critical values of $f$ since it contains also the set $B_\infty
(f)$ of {\it bifurcation points at infinity}. Roughly speaking,
$B_\infty(f)$ consists of points at which the restriction of $f$ to a neighbourhood of infinity (i.e. outside a
large enough ball) is not a locally
trivial bundle.  We say that $a\in B_\infty(f)$ is a
\emph{critical value at infinity} of  $f$. There are several
criteria to detect such a value,  one may consult e.g.  \cite{Su},
\cite{HL}, \cite{Ti-reg}, \cite{Du}, \cite{CK}, \cite{Ti-book}.
For instance, $a\in B_\infty(f)$ if and only if there exists a
sequence of points $(p_k)_{k\in \N}\subset \C^2$ such that $\|
p_k\| \to\infty, {\rm grad}\ f(p_k)\to 0$ and $f(p_k)\to a$ as
$k\to \infty$.

Upper bounds for $\# B_\infty(f)$ have been found in the 1990's by L\^e V.T. and M. Oka \cite{LO} in terms of Newton polyhedra at infinity.
 An estimation in terms of the degree $d$ of $f$ was given by Gwo\'zdziewicz and P\l oski \cite{GP}: if $\dim \Sing f \le 0$ then $\# B_\infty(f)\le \max \{ 1, d-3\}$.  In the general case (dropping the condition $\dim \Sing f \le 0$)
 we have  $\# B_\infty(f)\le d-1$, see
 e.g. \cite{jel1}, \cite{jel2}. Recently Gwo\'zdziewicz \cite{gw} proved the following estimation of $\# B_\infty(f)$: if
$\nu_0$ denotes the number of branches at infinity of the fibre
$f^{-1}(0)$, then the number of critical values at infinity other
than 0 is at most $\nu_0$. Here we refine and improve this
statement by using a different method, in which results by
Miyanishi \cite{miya}, \cite{miyanishi} and Gurjar \cite{gur} play
an important role.

 For $a\in \C$, let us denote by $\nu_a$ the number of branches at infinity of the fiber
$f^{-1}(a)$. This number is equal to $\nu_\gen$ for all values  $a\in \C$ except finitely
many for which one may have either
$\nu_a<\nu_\gen$ or $\nu_a>\nu_\gen$. Let $\nu_{\min} := \inf\{ \nu_a \mid
a\in \C \}$. Let us denote by  $b$ the number of \emph{points at infinity of $f$}, i.e. $b := \# \overline{f^{-1}(a)} \cap L_\ity$, where $L_\ity$ is the line at infinity $\P^n \m \C^n$.

Under these notations, our main result is the following:

\begin{theo}\label{1}
Let $f:\C^2\to \C$ be a polynomial function of degree $d$. Then:

\begin{enumerate}
 \item $\# B_\infty(f) \le \min \{\nu_\gen,\nu_{\min} + 1 \}$.

\item $\# \{ a\in \C \mid \nu_a< \nu_\gen \} \le \nu_\gen - b$.

\item $\# \{ a\in \C\mid\nu_a> \nu_\gen \} \le \nu_{\min}$.
\end{enumerate}

 In case  $\nu_\gen > \frac{d}{2}$, we moreover have:

\begin{enumerate}

\item[(d)]  $\# B_\infty(f) \le \min \{\nu_\gen -1,\nu_{\min} \}$.

\item[(e)] $\# \{ a\in \C \mid \nu_a> \nu_\gen \} \le \nu_{\min}-1$.

\end{enumerate}

\end{theo}

\begin{re}\label{r:1}
Point (a) of Theorem \ref{1} is equivalent to  Gwo\'zdziewicz's \cite[Theorem 2.1]{gw}. His result is a by-product of the local study of pencils of curves of Yomdin-Ephraim type. Our method is totally different and allows us to prove moreover several new issues, namely  (b)--(e)  of Theorem \ref{1}.
\end{re}

\begin{re}\label{r:2}
 As Gwo\'zdziewicz remarks, his inequality \cite[Theorem 2.1]{gw} is "almost" sharp, i.e. not sharp by one.
 Our new inequality (d) improves by one the inequality (a) under the additional condition $\nu_\gen > \frac{d}{2}$,
 thus yields the sharp upper bound, as shown by the example $f:\C^2 \to \C$, $f(x,y) = x + x^2 y$, where $d = \deg f = 3$,  $\nu_{\min} =\nu_\gen = 2$, $b=2$ and $B_\infty(f) = \{ 0\}$ with $\nu_0 = 3$.

The same example shows that our estimations  (b) and (e) are also
sharp.
\end{re}

\vspace{4mm}
\noindent
{\it Acknowledgements.} The authors are grateful to
Professor R.V. Gurjar from Tata Institute and Professor K. Palka
from IMPAN for helpful discussions.

\section{Proof of Theorem \ref{1}}

We need here the important concept of \emph{affine
surfaces which contain a cylinder-like open subset} which was introduced by M. Miyanishi
\cite{miya}. Let us recall it together with some properties which we shall use.

\begin{defi}\cite{miyanishi}
Let $X$ be a normal affine surface. We say that $X$  contains a
\emph{cylinder-like open subset} $U$, if there exists a smooth curve
$C$ such that $U\cong \C\times C$.
\end{defi}

Let $X$ be as in the above definition and let $\pi : U\to C$ be
the projection. After \cite[p.194]{miyanishi},  the projection
$\pi$ has a unique extension to a $\C$-fibration $\rho : X\to \bar
C$, where $\bar C$ denotes the smooth completion of the curve $C$.
We have the following important result of Gurjar and Miyanishi:

\begin{theo}\label{m1}  \cite{miya} \cite{gur-miya} \cite{gur}
Let $X$ be a normal affine surface with a $\C$-fibration $f : X\to
B$, where $B$ is a smooth curve. Then:
\begin{enumerate}
 \item $X$ has at most cyclic quotient singularities.
\item Every fiber of $f$ is a disjoint union of curves isomorphic to
$\C$.
\item A component of a fiber of $f$ contains at most one singular
point of $X$. If a component of a fiber occurs with multiplicity
$1$ in the scheme-theoretic fiber, then no singular point of $X$
lies on this component. \fin
\end{enumerate}

\end{theo}

\begin{co}\label{c1}
Let $X$ be a normal affine variety, which contains a cylinder-like
open subset $U$. Then the set $X\setminus U$ is a disjoint union
of curves isomorphic to $\C$. Moreover, every connected component
$l_i$ of this set contains at most one singular point of $X$.
\fin
\end{co}

 Let $f:\C^2\to \C$ be a polynomial function in fixed affine coordinates and denote by
$\tilde{f}(x,y,z)$ the homogenization of $f$ by a new variable
$z$, namely $\tilde{f}(x,y,z)=f_d+ zf_{d-1}+...+z^d f_0$. Let $X
:=  \{ ([x:y:z],t)\in \P^2\times \C \mid \tilde{f}(x,y,z) =
tz^d\}$ be the closure in $\P^2 \times \C$ of the graph $\Gamma:=
\graph(f)\subset \C^2\times \C$. Then $X$ is a hypersurface and
the points at infinity of $f$ are precisely the set $\{a_1,...,
a_n\} := \{f_d=0\} \subset L_\ity$, where $L_\ity := \{z=0\}$
denotes the line at infinity. The projection $\pi : X \to \C$,
$(x,t)\mapsto t$,  is a proper  extension of $f$.

Let $ X'$ be the normalization of $X$. Composing the normalisation map with $\pi$ yields $\pi': X'\to \C$, which is also a proper extension of $f$. For simplicity we shall denote it
again by $f$ in the following.

\begin{lem}\label{l1}
 The set $X'\setminus \Gamma$ is a disjoint union of affine
lines, $l_1,...,l_r$, each line $l_i$ is isomorphic to $\C$.
 On each line $l_i$ there is at
most one singular point of $X'$. Moreover, $b\le r\le \nu_{\min}$.
\end{lem}

\begin{proof}
Let us choose a line $l\subset \Bbb P^2$ such that $l\cap \{
a_1,...,a_n\}=\emptyset$. Let $X_1:=(\Bbb P^2\setminus l)\times \C
\cap X'$. The surface $X_1$ is affine and $X_1'\setminus
\Gamma=\bigcup^r_{i=1} l_i$, where $X_1'$ denotes the
normalization of $X_1$. Since the surface $X_1'$ contains a
cylinder-like open subset $U\cong\C^2\setminus l=\C\times \C^*$,
the first part of our claim follows from Corollary \ref{c1}.
Next, the map $f$ restricted to $l_i$ is finite, hence
surjective. This implies that every fiber of $f$ has a branch at infinity which intersects
$l_i$. In particular $r\le \nu_{\min}$. The inequality
$r\ge b$ is obvious.
\end{proof}

\noindent
Denote by $f_i : l_i\cong \C\to \C$ the restriction of $f$ to $l_i$. It can be identified with a
one variable polynomial, the degree of which is equal to the number $\nu_i$ of branches of a
generic fiber of $f$ which intersect
$l_i$. In particular $\sum^r_{i=1} \nu_i=\nu_\gen$.
 The polynomial
$f_i$ of degree $\nu_i$ can have at most $\nu_i-1$ singular points
and their images by $f_i$ are bifurcation values at infinity for
$f$ (this can be easily deduced from \cite{er}). There might also
be at most one singular point of $X'$ which lies on $l_i$, and
then its image by $f$ is possibly a bifurcation value too. Summing
up, we get that $f$ can have at most $\nu_\gen$ critical values at
infinity, which shows one of the inequalities of point (a).
Moreover, the inequality $\nu_a<\nu_\gen$ is possible only if $a$
is a critical value of some polynomial $f_i$. This means that $\#
\{ a\in \C\mid\nu_a< \nu_\gen \} \le \sum^r_{i=1} (\nu_i-1)\le
\nu_\gen- r\le \nu_\gen-b$. This proves (b).

Let us assume now $\nu_a=\nu_{\min}$. We have $\nu_a \ge
\sum^r_{i=1}\#\{ x\in l_i \mid f_i(x)=a\}$ since in
every such point $x$ there is at least one branch at infinity of the fiber $f^{-1}(a)$.
 Note that if
$f_i(x)=a$ then $\ord_x (f_i-a) =\ord_x f_i'+1$. Thus:
\[ \#\{ x\in l_i \mid f_i(x)=a\} =  \sum_{ x\in
l_i, f_i(x)=a } [\ord_x (f_i-a)- \ord_x f_i'].\]
 We have clearly the equality
$\sum_{x\in l_i} \ord_x (f_i-a)=\nu_i$. Hence
\[\sum_{ x\in
l_i, f_i(x)=a } [\ord_x (f_i-a)- \ord_x f_i']= \nu_i - \sum_{
x\in l_i, f_i(x)=a } \ord_x f_i'.\]
 Since $\sum_{ x\in l_i} \ord_x
f_i'=\nu_i-1$ we have:
 $$ \nu_i - \sum_{ x\in l_i,
f_i(x)=a } \ord_x f_i' = 1 + \sum_{x\in l_i,
f_i(x)\not=a} \ord_x f_i'.$$

Note that:
  \[
            1+ \sum_{x\in l_i, f_i(x)\not=a} \ord_x
f_i'\ge \#\{ x\in l_i \mid f(x)\not=a, \mbox{ and either }
f'_i (x)=0\ \mbox{ or } \ x \in \Sing(X')\}.
           \]
 The number at the right side is greater or
equal to the number of critical values at infinity of $f$
different from $f^{-1}(a)$. Finally, taking the sum over all $i\in
\{1,\ldots, r\}$  we get $\# B_\infty(f)   \le \nu_{\min}+1$, which completes the
proof of (a).

To prove (c), note that if the fiber $f^{-1}(a)$ does not contain
singular points of $X'$, then the intersection multiplicity
$\overline{l_i}\cdot f^{-1}(a)$ is equal to $\nu_i= \deg f_i$, hence
the fiber $f^{-1}(a)$ has  at most $\nu_i$ branches on $l_i$. This
implies $\nu_a\le \nu_\gen$. Hence  $\# \{ a\in \C \mid \nu_a> \nu_\gen \} \le
r\le \nu_{\min}$.

To prove (d) and (e) it is enough to show that if $\nu_\gen>
\frac{d}{2}$, then at least one line $l_i$ does not contain
singular points of  $X'$. Let $d_i$ be the smallest positive
integer such that $d_il_i$ is a Cartier divisor in $X'$ (such a
number exists because   $X'$ has only cyclic singularities). Since
$l_i$ is smooth, we have that $d_i=1$ if and only if the line
$l_i$ does not contain any singular point of $X'$, see Lemma
\ref{ber} below. Now let $Z$ be the closure of $\Gamma$ in $\Bbb
P^2\times \Bbb P^1$ and let $Z'$ denote its normalization. We have clearly the inclusion
 $X'\subset Z'$. Let $\Pi : Z'\to \Bbb
P^2$ the first projection, where the second projection $Z'\to \Bbb
P^1$ is an extension of $f$ which we will denote for simplicity by
$f$ again. Let $f^{-1}(\infty)=S_1\cup... \cup S_k$ (where $S_i$
are irreducible and taken with reduced structure). Recall that
$L_\infty=\Bbb P^2\setminus \C^2$ is the line at infinity. We have
$\Pi^*(L_\infty)=\sum_{i=1}^k m_iS_i+\sum^r_{i=1}
e_i\overline{l_i}$. Since $\Pi^*(L_\infty)$ is a Cartier divisor
we have $e_i=n_id_i,$ where $n_i$ is a positive integer.

Let us assume that any line $l_i$  contains some singular point of
$X'$, i.e., that $d_i> 1$ for any $i$. Denoting by
$F\subset \Bbb P^2$ the closure of a general fiber of $f=a$, since $\Pi$ is a birational morphism, we have:
$$d= F\cdot L_\infty=\Pi^*(F)\cdot \Pi^*(L_\infty)=f^*(a)\cdot (\sum_{i=1}^k m_iS_i+\sum^r_{i=1}
e_i\overline{l_i}).$$

Note that $\Pi^*(F)\cdot \sum^k_{i=1} m_iS_i=0$ since
$|f^*(a)|\cap |\sum^k_{i=1} m_iS_i|=|f^*(a)|\cap
|f^*(\infty)|=\emptyset$. Moreover we have $\nu_i=f^*(a) \cdot
\overline{l_i}$. Thus:
 $$d=\sum^r_{i=1}
n_id_i\nu_i\ge \sum^r_{i=1} 2\nu_i=2\nu_\gen$$
and this ends our proof.
 \fin

\begin{lem}\label{ber}
Let $X^n$ be an algebraic variety and let $Z^r\subset X^n$ be a
subvariety which is a complete intersection in $X^n$. If a point
$z\in Z^r$ is nonsingular on $Z^r$, then it is nonsingular on $X^n$.
\end{lem}

\begin{proof}
 Let us recall that if $X^n$ is an algebraic variety, then
a point $z\in X^n$ is smooth, if and only if the local ring ${\cal
O}_z(X)$ is regular. This is equivalent to the fact  that $\dim_\C
{\mathfrak m}/{\mathfrak m}^2=\dim X^n=n,$ where $\mathfrak m$
denotes the maximal ideal of ${\cal O}_z(X^n)$.

Let $I(Z^r)=\{ f_1,...,f_{l}\}\subset \C[X^n],$ where $l=n-r$,  be
the ideal of $Z^r$ in the ring $\C[X^n]$. We have ${\cal
O}_z(Z^r)={\cal O}_z(X^n)/(f_1,...,f_l)$. In particular if
$\mathfrak m'$ denotes the maximal ideal of ${\cal O}_z(Z^r)$ and
$\mathfrak m$ denotes the maximal ideal of ${\cal O}_z(X^n)$ then
$\mathfrak m' = \mathfrak m/(f_1,...,f_l)$. Let
$\alpha_i$ denote a class of the polynomial $f_i$ in ${\mathfrak
m}/{\mathfrak m}^2$. Let us note that
\begin{equation}\label{eq}
{\mathfrak m'}/{\mathfrak m'}^2={\mathfrak m}/({\mathfrak
m}^2+(\alpha_1,...,\alpha_l)).
\end{equation}
Since the point $z$ is smooth on $Z^r$ we have $\dim_\C {\mathfrak
m'}/{\mathfrak m'}^2=\dim Z^r=\dim X^n-l$. Take a basis
$\beta_1,...,\beta_{n-l}$ of the space $ {\mathfrak m'}/{\mathfrak
m'}^2$ and let $\overline{\beta_i}\in {\mathfrak m}/{\mathfrak
m}^2$ correspond to $\beta_i$ under the correspondence (\ref{eq}).
Note that vectors $\overline{\beta_1},...,\overline{\beta_{n-l}},
\alpha_1,..., \alpha_l$ generate the space ${\mathfrak
m}/{\mathfrak m}^2$. This means that $\dim_\C {\mathfrak
m}/{\mathfrak m}^2\le n-l+l=n=\dim X^n$. Since we also have the general inequality
$\dim_\C {\mathfrak m}/{\mathfrak m}^2\ge \dim X^n$, we obtain the equality
$\dim_\C {\mathfrak m}/{\mathfrak m}^2=\dim X^n$, which means that $X^n$ is nonsingular at $z$.
\end{proof}

\end{document}